\documentclass[11pt]{article}

\usepackage{amsmath,amssymb}
  \usepackage{graphics} 
  \usepackage{epsfig} 

\newtheorem{theorem}{Theorem}[section]
\newtheorem{corollary}{Corollary}
\newtheorem{lemma}[theorem]{Lemma}

\newtheorem{proof}{Proof}
\newtheorem{definition}[theorem]{Definition}
\newtheorem{remark}{Remark}

\title{A Gaussian quadrature rule for oscillatory integrals on a bounded interval}
\date{\today}

\author{A. Asheim, A. Dea\~no, D. Huybrechs, and H. Wang{\footnote{\{andreas.asheim,alfredo.deano,daan.huybrechs,haiyong.wang\}@cs.kuleuven.be}} {\footnote{
This research was supported by the Fund for Scientific Research Flanders through Research Project G.0617.10. A. Asheim is supported by the Norwegian Research Council's Espedal Fellowship.}}\\[5mm]
Department of Computer Science\\[2mm]
 University of Leuven, Belgium}

\begin{document}

\maketitle





\begin{abstract}
We investigate a Gaussian quadrature rule and the corresponding orthogonal polynomials for the oscillatory weight function $e^{i\omega x}$ on the interval $[-1,1]$. We show that such a rule attains high asymptotic order, in the sense that the quadrature error quickly decreases as a function of the frequency $\omega$. However, accuracy is maintained for all values of $\omega$ and in particular the rule elegantly reduces to the classical Gauss-Legendre rule as $\omega \to 0$. The construction of such rules is briefly discussed, and though not all orthogonal polynomials exist, it is demonstrated numerically that rules with an even number of points are always well defined. We show that these rules are optimal both in terms of asymptotic order as well as in terms of polynomial order.
\end{abstract}



\bigskip

Highly oscillatory integrals are ubiquitous in applied mathematics. The numerical evaluation of such integrals has been a rich field of study in the last decade, following a series of hallmark papers by Iserles \& N\o{}rsett \cite{Iserles:2004us,Iserles:2005tj,Iserles:2006fi,Iserles:2004ke,Iserles:2006uh}. Prompted by the unexpected succes of modified Magnus expansions for oscillatory differential equations \cite{Iserles:2002gg}, these authors introduced asymptotic and Filon-type methods as a generalisation of the classical Filon method \cite{filon1928quadrature}. Other innovations in this field include the numerical method of steepest descent \cite{Huybrechs:2006gi}, based on numerical evaluation of steepest descent integrals, and Levin-type methods \cite{Levin:1996kp,Olver:2006jp,Olver:2010jj}, based on solving an associated differential equation. A common property of all these methods is that high asymptotic order can be 
attained, by which we mean an error that behaves like $\mathcal{O}(\omega^{-\alpha})$, $\omega \to\infty$, for a positive $\alpha$. Unlike classical asymptotic expansions, whose error behaves similarly as $\omega\to\infty$, these methods are in principle not limited in accuracy for a fixed $\omega$.

This work concerns an \emph{asymptotically optimal} oscillatory quadrature rule that is valid for all frequencies. The discussion will focus on the Fourier-integral
\begin{equation}\label{eq:HOI}
I[f]:=\int_{-1}^1f(x)e^{i\omega x}{\rm d}x,
\end{equation}
where we shall consider $\omega$ ranging from $0$ to $\infty$. Of the above mentioned quadrature rules, the numerical method of steepest descent can be considered to be asymptotically optimal; it delivers an error of size $\mathcal{O}(\omega^{-2n-1})$, when using $2n$ quadrature points in the complex plane for this integral. Filon-type methods are based on interpolation of the amplitude function $f$. By choosing interpolation points that scale towards the endpoints like $1/\omega$, they deliver $\mathcal{O}(\omega^{-n-1})$ error with the same number of points \cite{Iserles:2004ke}. In \cite{Huybrechs:2012gm} it was pointed out that by choosing the interpolation points for the Filon-type method to be precisely the quadrature points in the complex plane of the numerical steepest descent method, the so-called \emph{superinterpolation points}, also here an error of $\mathcal{O}(\omega^{-2n-1})$ can actually be attained with only $2n$ points.

Next, one could ask, how well these methods fare when $\omega$ is small. Unfortunately, the superinterpolation points are unbounded in the limit $\omega\to0$. As such, they are not a good choice for small $\omega$. For $\omega=0$ the Gauss-Legendre points are optimal in the sense that they integrate exactly polynomials of degree $2n-1$ using $n$ points. As for Filon-type methods, the interpolation points can be chosen such that they approach the Legendre points as $\omega \to 0$ \cite{Iserles:2004ke}. However, it is not clear what is an optimal way to do so. In the context of the so-called exponential fitting methods, such strategies have been described with relatively few quadrature points \cite{Ixaru:2001fe}, or with an heuristic dependence on $\omega$ \cite{Ledoux:2012ww}. However, in both cases, using $2n$ 
points still yields an error of size $\mathcal{O}(\omega^{-n-1})$, which is not optimal.

In a recent work, \cite{integraltransforms}, it was shown that Gaussian quadrature rules can be constructed for certain oscillatory integral transforms, i.e., integrals over unbounded domains. This goes against common wisdom, saying that Gaussian rules can be found for positive weight functions only. A more correct statement is that existence and uniqueness proofs, as well as many construction methods, rely on positive weight functions. For integral transforms of the form 
\[
\int_0^\infty f(x)e^{i\omega x}{\rm d}x,
\]
it was shown in \cite{integraltransforms} that applying the Gauss-Laguerre rule along the path of steepest descent yields a Gaussian rule, which is exact for $f$ being a polynomial of degree $\leq 2n-1$. Moreover, a clear connection was shown between polynomial accuracy and asymptotic accuracy for such integrals, where higher polynomial accuracy means higher asymptotic accuracy. Thus, the Gaussian rule for this integral transform attains an error $\mathcal{O}(\omega^{-2n-1})$. Similar results were shown for other oscillators, all resulting in rules with nodes in the complex plane.

The topic of the current work is a further investigation of Gaussian rules for oscillatory weight functions, but now on a bounded domain. This, we will see, is a more difficult case, since it is no longer possible to remove the dependency on $\omega$ by a simple rescaling as in the unbounded case. For integral \eqref{eq:HOI}, the oscillatory part $e^{i\omega x}$ is the weight function, and one thus seeks rules that integrate polynomials up to degree $2n-1$ exactly. This endeavour poses several problems. The rules depend on $\omega$ in a non-trivial way, and must be computed numerically. One cannot guarantee existence of the rules, though numerical evidence indicates that all rules with an even number of points exist. On the other hand, the resulting rules reduce elegantly to Gauss-Legendre by construction in the limit $\omega\to 0$. Moreover, it will be proved, under mild assumptions, that the quadrature points tend to the superinterpolation points in the high-frequency limit $\omega\to \infty$. 
This implies that the method yields optimal asymptotic order in addition to optimal polynomial order. These observations are the basis of the statement that this Gaussian rule is truly optimal for oscillatory integrals of the form \eqref{eq:HOI} throughout the frequency regime.

This paper is built up as follows. In \S\ref{S:prel}, we recall some preliminary facts regarding Gaussian quadrature and highly oscillatory quadrature methods. In \S\ref{S:experiments}, analytic expressions for the cases $n=1$ and $n=2$, as well as several numerical experiments, shed light on the questions of existence and other properties of Gaussian rules for integrals of the form \eqref{eq:HOI}. This section is concluded with a set of conjectures, most of which are proved in \S\ref{S:properties}, on the properties of the orthogonal polynomials, and \S\ref{S:4}, on the properties of the quadrature rule.

\section{Preliminaries}\label{S:prel}

\subsection{Gaussian quadrature}

An $n$-point Gaussian quadrature rule $\{x_j,w_j\}_{j=1}^n$ is a quadrature rule which has optimal polynomial accuracy, i.e., it is exact for all polynomials of degree $\leq 2n-1$, 
\[
\int_a^b f(x)h(x){\rm d}x = \sum_{j=1}^n w_jf(x_j),\qquad f\in \mathcal{P}_{2n-1}.
\]
It is well known that the nodes of a Gaussian rule are precisely the zeros of the $n$-th orthogonal polynomial with respect to the weight function $h$ \cite{gautschi2004orthogonal}. Classical theory of Gaussian quadrature ensures that such rules exist and that they are uniquely defined whenever $h$ is positive. Moreover, the positivity of the weight function also ensures that $x_j\in [a,b]$, $j=1,\hdots,n$, and that the weights $w_j$ are all positive.

The monic orthogonal polynomial $p_n$ can be computed in various ways, for example by the recurrence 
\begin{equation}\label{eq:recurrence}
p_{k+1}(x)=(x-\alpha_k)p_k(x)-\beta_kp_{k-1}(x),
\end{equation}
with initial values $p_{-1}=0$ and $p_0=1$. Defining the pairing
\[
(f,g):=\int_a^b f(x)g(x)h(x){\rm d}x,
\]
which is an inner product if $h$ is positive, the recurrence coefficients are given by
\begin{equation}\label{eq:reccoeffs}
\alpha_k=\frac{(xp_k,p_k)}{(p_k,p_k)}
,\quad
\beta_k=\frac{(p_k,p_k)}{(p_{k-1},p_{k-1})}.
\end{equation}
Alternatively, writing $p_n(x)=x^n+\sum_{k=0}^{n-1} a_k x^k$, the coefficients $a_0,\hdots, a_{n-1}$ satisfy the linear system
\begin{equation}\label{eq:asystem}
 \begin{pmatrix}
 \mu_0 & \mu_1 & \hdots &\mu_{n-1} \\
 \mu_1 & \mu_2 & \hdots &\mu_{n} \\
 \vdots & \vdots & \ddots &\vdots \\
 \mu_{n-1} & \mu_n & \hdots &\mu_{2n-2} \\
 \end{pmatrix}
  \begin{pmatrix}
	a_0\\a_1\\ \vdots \\ a_{n-1}
 \end{pmatrix} 
  =
  -\begin{pmatrix}
	\mu_n\\ \mu_{n+1} \\ \vdots \\ \mu_{2n-1}
	 \end{pmatrix},
\end{equation}
where the moments $\mu_m$ are defined as
\[
 \mu_m:=\int_a^b x^mh(x){\rm d}x.
\]

\subsection{Gaussian quadrature and the method of steepest descent}

In the method of steepest descent, an oscillatory integral of the form \eqref{eq:HOI} is rewritten along the so-called paths of steepest descent. Assuming $f$ is analytic, we may deform the path of integration as follows:
\begin{equation}\label{eq:sd}
\int_{-1}^1f(x)e^{i\omega x}{\rm d}x =\frac{ie^{-i\omega}}{\omega}\int_0^\infty f(-1+it\omega^{-1})e^{-t}{\rm d}t-\frac{ie^{i\omega}}{\omega}\int_0^\infty f(1+it\omega^{-1})e^{-t}{\rm d}t.
\end{equation}
Applying the Gauss-Laguerre quadrature on the two resulting integrals yields an error that decays like $\mathcal{O}(\omega^{-2n-1})$, when using $n$ points for each integral (so $2n$ points in total) \cite{Huybrechs:2006gi}. The Filon-type method, due to Iserles \& N\o rsett\ \cite{Iserles:2005tj}, is based on polynomial interpolation of the amplitude function $f$. In \cite{Huybrechs:2012gm} it was shown that using the complex points obtained by applying Gaussian quadrature on the steepest descent integrals in \eqref{eq:sd} as interpolation points in a Filon-type method, also yields an error decay of $\mathcal{O}(\omega^{-2n-1})$. Following the terminology of \cite{Huybrechs:2012gm}, these points are referred to as superinterpolation points:

\begin{definition}\label{def:superinterpolation}
Let $\{\xi_j,\eta_j\}_{j=1}^n$ denote the $n$ point Gauss-Laguerre quadrature rule. The corresponding $2n$ superinterpolation points are defined as
\[
\left\{-1+ \frac{i \xi_j}{\omega}\right\}\cup \left\{1+ \frac{i \xi_j}{\omega}\right\}_{j=1}^n.
\]
\end{definition}
The corresponding quadrature weights for these points are
\[
\left\{\frac{\eta_j}{\omega}\right\}\cup \left\{\frac{\eta_j}{\omega}\right\}_{j=1}^n.
\]
Note that Filon-type methods have the advantage that accuracy can be increased by adding interpolation points wherever they are needed, while maintaining asymptotic accuracy. This can be used to minimise the effect of eventual singularities in the complex plane. This leads, of course, to different weights. Also note that the superinterpolation points are not well defined in the limit $\omega\to0$.

\section{Polynomials orthogonal to $e^{i\omega x}$ on $[-1,1]$, analytic examples and numerical experiments}\label{S:experiments}

Considering polynomials orthogonal w.r.t. the weight $e^{i\omega x}$ on $[-1,1]$, the non-positive weight function does not enable us to use classical existence and uniqueness theory. However, assuming existence for the time being, at least for some $n$ and $\omega$, we denote the $n$-th monic orthogonal polynomial by $p_n^\omega$. The polynomial $p_n^\omega$ is orthogonal in the sense that,
\[
\int_{-1}^1p_n^\omega(x)x^je^{i\omega x}{\rm d}x=0, \qquad j=0,1,\hdots,n-1, \quad \omega\geq 0.
\]
The moments can be computed explicitly,
\[
\mu_m^\omega=\int_{-1}^1 x^me^{i\omega x}{\rm d}x = (-1)^m(i \omega )^{-1-m}(\Gamma(1+m,-i \omega )-\Gamma(1+m,i \omega )),
\]
where $\Gamma(z,a)$ is the incomplete Gamma functions \cite{olver2010nist}, or by the recurrence
\[
\mu_m^\omega=\frac{e^{i\omega}-(-1)^me^{-i\omega}}{i\omega}-\frac{m}{i\omega}\mu_{m-1}^\omega,\qquad \mu_0^\omega=\frac{2\sin\omega}{\omega},
\]
which is obtained through integration by parts. This recurrence should be used with great care, since it's forward stable only when $m<\omega$. For $m>\omega$, the backward recursion should be used. Assuming existence of these polynomials, they can be computed using the linear system for the coefficients \eqref{eq:asystem}. For the cases $n=1$ and $n=2$ the polynomials can be computed analytically, and this gives some insight into the nature of the higher order rules.

\subsection{The case $n=1$}\label{ss:nis1}

In the case $n=1$ the orthogonal polynomial takes the form $p_1^\omega(x)=x+a_0$, where 
\begin{equation}\label{eq:n1}
a_0 = -\frac{\mu_1^\omega}{\mu_0^\omega} = \frac{i}{\tan \omega}-\frac i \omega.
\end{equation}
Note that the single quadrature point $-a_0$ remains on the imaginary axis. The rule is undefined when $\omega$ is a multiple of $\pi$, except in the limit $\omega\to 0$ where $a_0\to 0$. This limit corresponds to the Gauss-Legendre rule.

\subsection{The case $n=2$}

\begin{figure}
\begin{center}
\parbox{\textwidth}{
\includegraphics[width=\columnwidth]{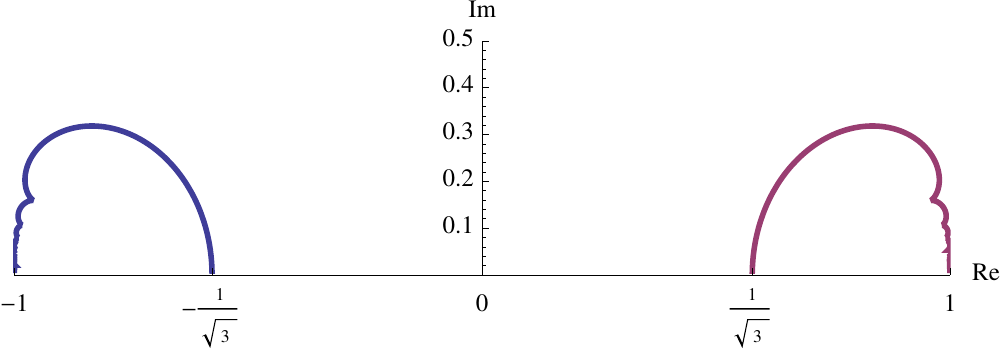}
\caption{Two Gaussian points for the oscillatory integral start out at the Gauss-Legendre points, $\pm 1/\sqrt{3}$, for $\omega=0$ and follow these curves in the complex plane for increasing $\omega$.}\label{Fig:2Gausspoints}
}
\parbox{\textwidth}{
\includegraphics[width=\columnwidth]{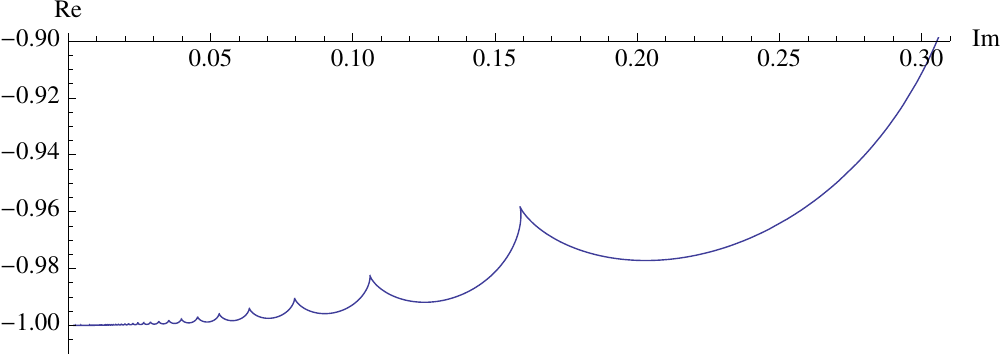}
\caption{Reflecting and zooming in on the left curve in Figure \ref{Fig:2Gausspoints}.}
}
\end{center}
\end{figure}

In the case $n=2$, one obtains for $p_2^\omega(x)=x^2+a_1 x + a_0$ the coefficients
\[
a_0=\frac{2+3 \omega ^2-2 \omega ^4+\left(-2+\omega ^2\right) \cos (2\omega)-4 \omega  \sin (2\omega)}{\omega ^2 \left(-1+2 \omega ^2+\cos (2\omega)\right)}.
\]
\[
a_1=-\frac{2 i \left(-2+2 \omega ^2+2 \cos (2\omega)+\omega  \sin (2\omega)\right)}{\omega  \left(-1+2 \omega ^2+\cos (2\omega)\right)},
\]
In the low-frequency limit, for $\omega\to0$, we recover the Gauss-Legendre rule again. Observe that
\[
\begin{aligned}
a_0&=-\frac{1}{3}-\frac{2}{45}\omega^2+\mathcal{O}(\omega^4),\\
a_1&=-\frac{4i}{15}\omega+\frac{4i}{1575}\omega^3+\mathcal{O}(\omega^3).\\
\end{aligned}
\]
Since the roots $x_\pm$ satisfy $-x_{+}-x_{-}=a_1$ and $x_+x_-=a_0$, it follows that $x_{\pm}=\pm\frac1{\sqrt{3}} +\mathcal{O}(\omega)$. They are a perturbation of the Gauss-Legendre nodes as expected.

For $\omega\to\infty$, we observe that
\[
\begin{aligned}
a_0&=-1+\frac{2\cos^2 2\omega}{\omega^2}-\frac{2\sin 2\omega}{\omega^3}
+\mathcal{O}(\omega^{-4}),\\
a_1&=-\frac{2i}{\omega}-\frac{i\sin 2\omega}{\omega^2}
+\frac{2i\sin^2 2\omega}{\omega^3}+\mathcal{O}(\omega^{-4}),
\end{aligned}
\]
leading to $x_\pm=\pm 1+\frac{i}{\omega} +\mathcal{O}(\omega^{-2})$. These turn out to be the two-point superinterpolation points as defined in Definition \ref{def:superinterpolation} -- recall that $1$ is the single root of the Laguerre polynomial of degree $1$.

The roots of the polynomials, which are the two quadrature points for the Gaussian rule, are given explicitly by
\begin{multline*}
x_{\pm}=\left(\omega  \left(-1+2 \omega ^2+\cos 2\omega\right)\right)^{-1} \Big{[} i \left(-2+2 \omega ^2+2 \cos 2\omega+\omega  \sin 2\omega\right)\\
 \pm  \sqrt{-3+6 \omega ^2-12 \omega ^4+4 \omega ^6+\left(4-6 \omega ^2\right) \cos 2\omega-\cos 4 \omega+4 \omega ^3 \sin 2\omega}\Big{]}.
\end{multline*}
Although this expression is rather complicated, one sees that the points exist for all $\omega$, unlike in the case $n=1$. Fig. \ref{Fig:2Gausspoints} shows the curves that the two quadrature points trace out in the complex plane as $\omega$ increases. The qualitative behaviour seen in this figure also appears for higher $n$. The points leave the real line and drift into the complex plane orthogonal to the real line. After an excursion into the complex plane the points attract in a rather irregular manner towards the two endpoints of the interval. The curves appear to have cusps for certain values of $\omega$. Regarding the curves as parametric curves, it is clear that a cusp can only happen at a singular point, i.e., where $\frac{\partial x_{\pm}}{\partial \omega}$ vanishes.

\begin{figure}
\begin{center}
\includegraphics[width=.7\columnwidth]{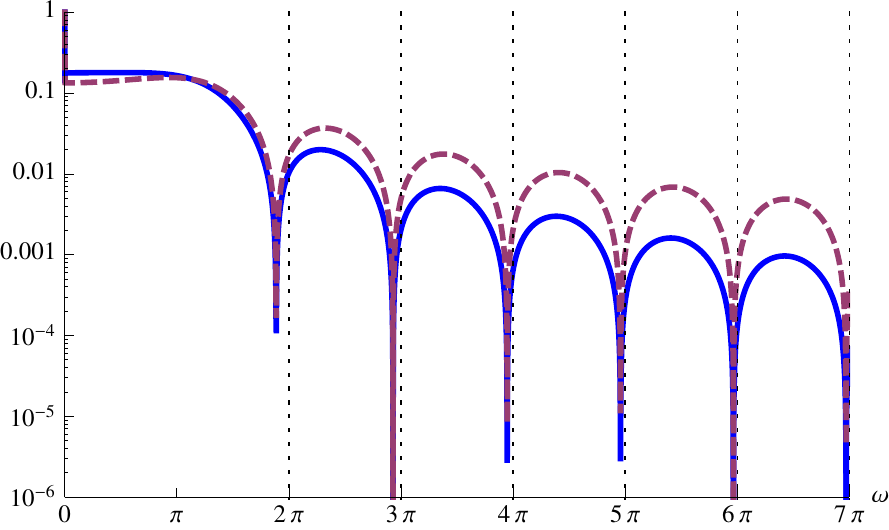}
\caption{The absolute value of $(p_2^\omega,p_2^\omega)$ (continuous) and $\frac{\partial x_1}{\partial \omega}$ (dashed) as a function of $\omega$.}\label{Fig:norm}
\end{center}
\end{figure}

In order to compute the recurrence coefficients, defined in \eqref{eq:reccoeffs}, that would enable us to compute $p_3^\omega$, we need the quantity
\[
(p_2^\omega,p_2^\omega)=\int_{-1}^1 (p_2^\omega(x))^2e^{i\omega x}{\rm d}x=-\frac{16 \left(2 \omega^3 \cos(\omega)+\omega^2 \left(-3+\omega^2\right) \sin(\omega)+\sin(\omega)^3\right)}{\omega^5 \left(-1+2 \omega^2+\cos(2 \omega)\right)}.
\]
For certain values of $\omega$, this expression vanishes and the recurrence coefficient $\alpha_2$ does not exist. The first such value occurs near $\omega=2\pi$, and this value can be numerically computed to be $\omega^*_1=5.92966\hdots$. The values of $\omega$ for which $(p_2^\omega,p_2^\omega)=0$ are difficult to compute explicitly, but one can give some information for large values of $\omega$. Dividing by $\omega^4$ throughout, we obtain that $(p_2^\omega,p_2^\omega)=0$ is equivalent to
\[
 \frac{2\cos(\omega)}{\omega}+\left(1-\frac{3}{\omega^2}\right) \sin(\omega)+\frac{\sin(\omega)^3}{\omega^4}=0.
\]
Hence, to leading order we have the zeros of $\sin\omega$ as solutions, so $\omega_k=k\pi+\mathcal{O}(k^{-1})$, $k\to\infty$. The first correction to this estimation gives
\[
\omega_k=k\pi-\frac{2}{k\pi}+\mathcal{O}(k^{-3})
\]
This is illustrated in Figure \ref{Fig:norm}, where also $|\frac{\partial x_+}{\partial \omega}|$ is plotted. The figure shows that the zeros of $(p_2^\omega,p_2^\omega)$ appear to coincide with the zeros of $|\frac{\partial x_+}{\partial \omega}|$, and thus these values of $\omega$ correspond to the the cusps in Fig. \ref{Fig:2Gausspoints}. We may conclude from this that there is a breakdown of the recurrence \eqref{eq:recurrence} for countably many values of $\omega$, and that these problematic values correspond to cusps in the curves $x_j(\omega)$. As a result, $p_3^\omega$ is undefined for these special values of $\omega$, just like $p_1^\omega$ was undefined for values of $\omega$ that are exact multiples of $\pi$. However, by avoiding the recurrence and by computing with the moments as in eq. \eqref{eq:asystem}, for example, one can still compute the orthogonal polynomial of degree $4$ for all values of $\omega$.

\begin{figure}
\begin{center}
\includegraphics[width=.7\columnwidth]{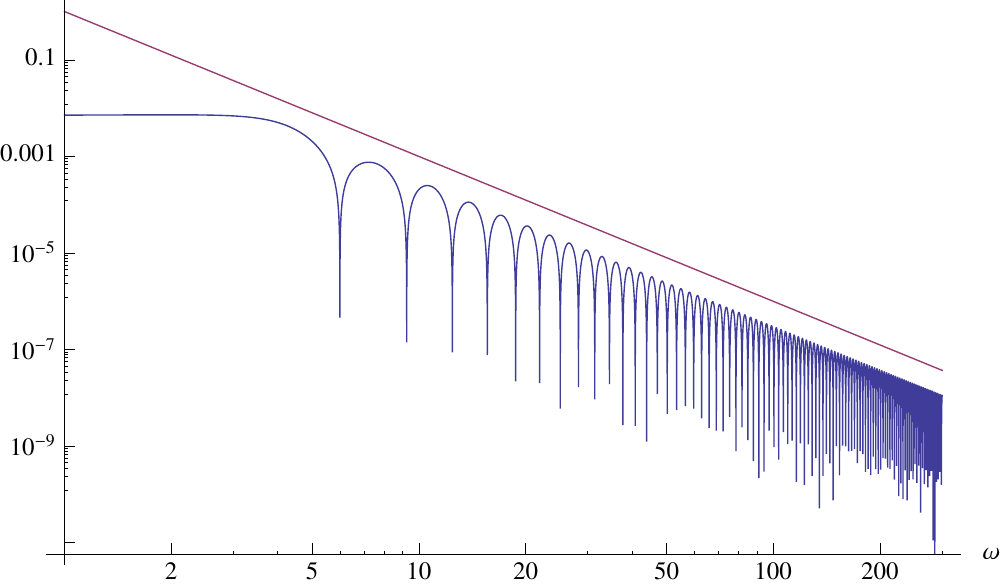}
\caption{Absolute error of the rule with two Gaussian points for the integral $\int_{-1}^1\sin(x)e^{i\omega x}{\rm d}x$. The straight line shows $\omega^{-3}$.}\label{Fig:2Gausspoints_exp}
\end{center}
\end{figure}

The quadrature method obtained from the superinterpolation points has asymptotic order $3$, and this method presumably has the same order. Applying it to the integral
\[
\int_{-1}^1\sin(x)e^{i\omega x}{\rm d}x,
\]
the error is shown in Fig. \ref{Fig:2Gausspoints_exp}. We see that the resulting method indeed appears to have asymptotic order $3$. Note that a Filon-type method with two quadrature points is in general only expected to have asymptotic order $2$. This can be achieved by using the endpoints $\pm 1$, or any two points that move towards $\pm 1$ at a rate of $1/\omega$ \cite{Iserles:2004ke}.

\subsection{Numerical experiments for $n>2$}

\begin{figure}
\begin{center}
\parbox{\textwidth}{
\includegraphics[width=\columnwidth]{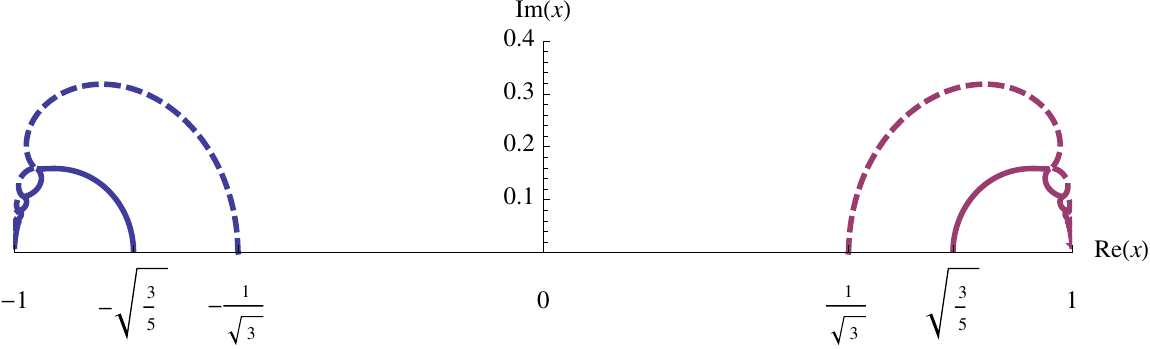}
\caption{Three Gaussian points for the oscillatory integral start out at the Gauss-Legendre points for $\omega=0$ and follow these curves in the complex plane for increasing $\omega$ (solid lines). The third point is always purely imaginary and is not shown. The dashed lines are the trajectories of the roots of $p_2^\omega$, as in Fig. \ref{Fig:2Gausspoints}.}\label{Fig:3Gausspoints}
}
\parbox{\textwidth}{
\includegraphics[width=\columnwidth]{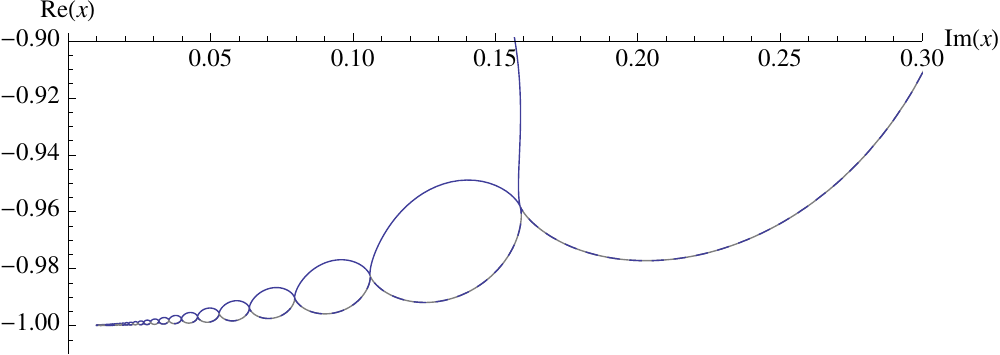}
\caption{Reflecting and zooming in on the left curves in Figure \ref{Fig:3Gausspoints} near $-1$.}
}
\end{center}
\end{figure}

For $n>2$, obtaining closed forms expressions for zeros of the orthogonal polynomials is highly impractical, if at all possible. Turning to numerics, the coefficients of the monic polynomial can be computed from the Hankel system \eqref{eq:asystem}, and the roots can be found numerically. Note that this system is likely to be ill-conditioned. The computations in this paper have been carried out in high precision in Maple.

The roots for $n=3$ behave in a similar manner as in the case $n=2$. In Fig. \ref{Fig:3Gausspoints}, we see that the cusps in this case seem to coincide with the cusps in the case $n=2$. However, there is an extra root which is always on the imaginary axis, a consequence of the symmetry of the polynomials with respect to the imaginary axis. In Fig \ref{Fig:3Gausspoints2}, one observes that this root is indeed undefined for a set of discrete values $\omega$ identified as the cusps in the case $n=2$.

\begin{figure}
\begin{center}
\includegraphics[width=.7\columnwidth]{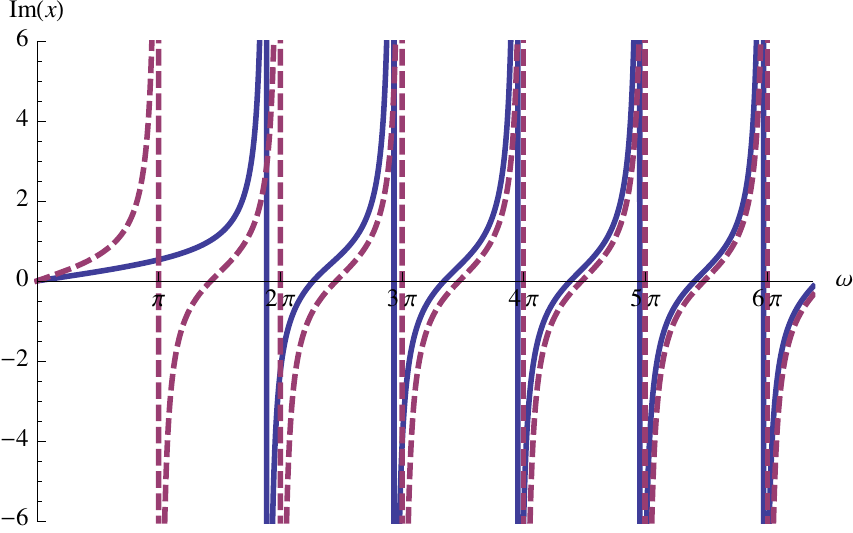}
\caption{The size of the imaginary root of $p_3^\omega$. The size of the root of $p_1^\omega$, Eq. \eqref{eq:n1}, is shown as a dashed curve.}\label{Fig:3Gausspoints2}
\end{center}
\end{figure}

One can also compute zeros of polynomials of higher order. For $n=16$, the zeros behave qualitatively in much the same way as the zeros of $p_{2}^\omega$. The cusps in the curves correspond to the same critical values of $\omega$ regardless of the root, but these values are different from those in the case $n=2$. The next experiment concerns the behaviour of the nodes for large $\omega$. Figure \ref{Fig:16Gausspoints2} shows the difference between the $8$ roots of $p_{16}^\omega$ near $-1$ and the $8$ corresponding superinterpolation points. It appears that the difference goes like $\mathcal{O}(\omega^{-2})$, similar to what was established for the case $n=2$. For large values of $\omega$, the roots of the orthogonal polynomial tend to the superinterpolation points.
\begin{figure}
\begin{center}
\includegraphics[width=.9\columnwidth]{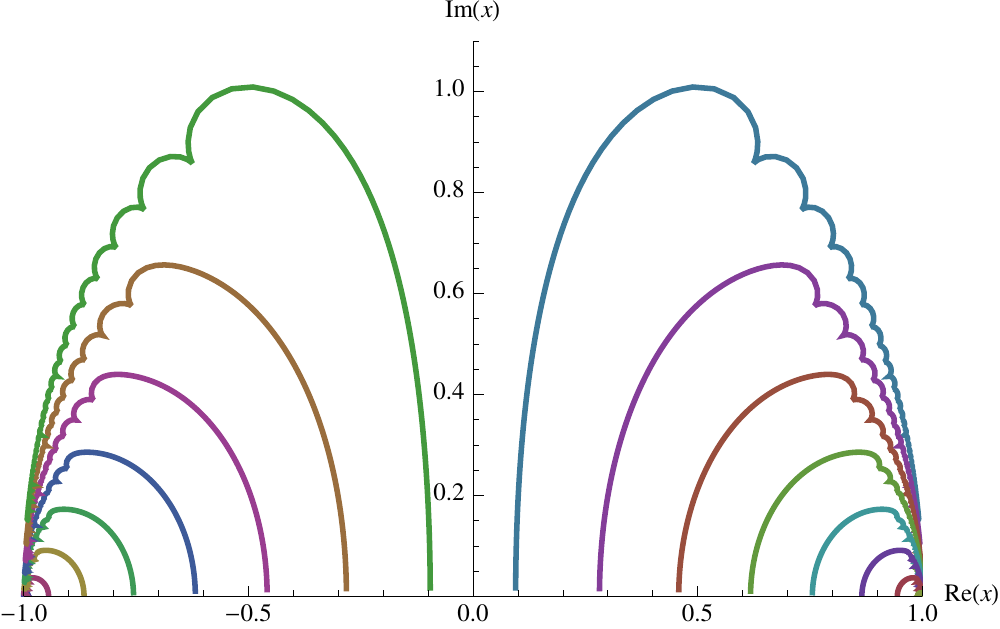}
\caption{Sixteen Gaussian points for the oscillatory integral start out at the Gauss-Legendre points on $[-1,1]$ for $\omega=0$ and follow these curves in the complex plane for increasing $\omega$.}\label{Fig:16Gausspoints}
\end{center}
\end{figure}

\begin{figure}
\begin{center}
\includegraphics[width=.9\columnwidth]{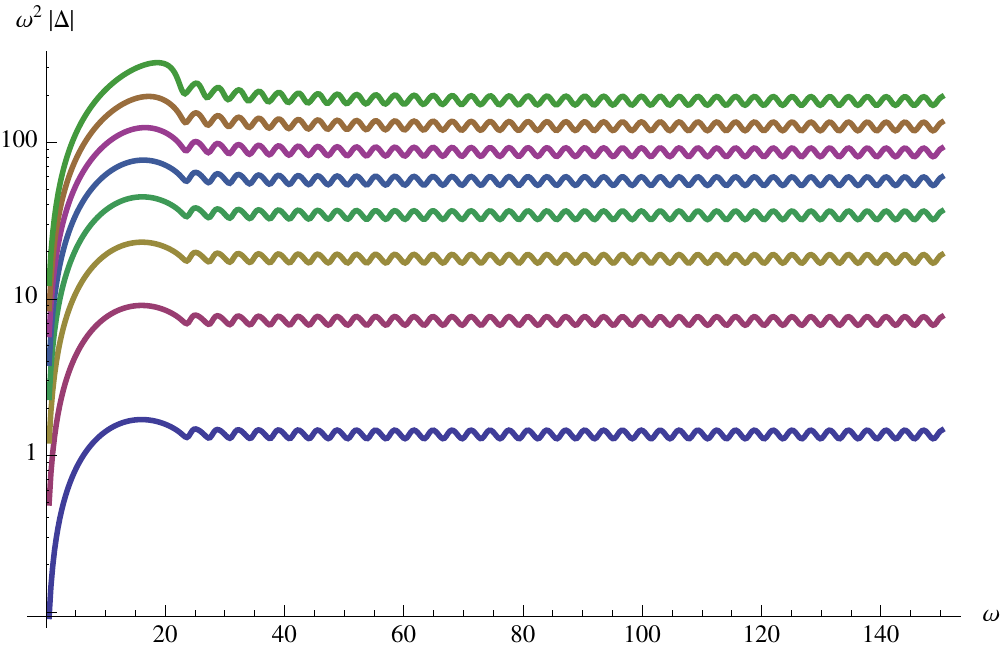}
\caption{The absolute difference between the first $8$ roots of $p_{16}^\omega$ and the corresponding superinterpolation nodes, scaled by $\omega^2$.}\label{Fig:16Gausspoints2}
\end{center}
\end{figure}

\subsection{Conjectures based on observations}

From the above discussion we can list several conjectures:
\begin{enumerate}
\item Rules with an even number of points exist for all $\omega$.
\item Rules with an odd number of points are not defined for all $\omega$.
\item In particular, $p_{2n+1}^\omega$ is not defined for those critical values of $\omega$ that correspond to the cusps of the roots of $p_{2n}^\omega$.
\item The polynomials $p_{n}^\omega$ are symmetric with respect to the imaginary axis.
\item The roots of $p_{2n}^\omega$ approach the $2n$ superinterpolation points at a rate of $\mathcal{O}(\omega^{-2})$ in the high-frequency limit.
\item The Gaussian rule based on the zeros of $p_{2n}^\omega$ has asymptotic order $\mathcal{O}(\omega^{-2n-1})$.
\end{enumerate}
In this paper we shall not prove the first conjecture, on the existence of even-degree polynomials, but assume it to be true. We shall prove conjectures 3 (Lemma \ref{lem:symmetry}) and 4 in \S\ref{S:properties} and conjectures 5 and 6 in \S\ref{S:4}.

\section{Properties of the orthogonal polynomials}\label{S:properties}

In this section we set out to describe a number of interesting properties of the orthogonal polynomials that are useful later on to explain the observations that were made in the previous section.

\subsection{Symmetry}
Symmetries of weight functions and intervals lead to symmetries of the corresponding polynomials. The complex exponential weight function has symmetries with respect to the imaginary axis, in the sense that  $w(z) = \overline{w(-\overline{z})}$. Note that the point $-\overline{z}$ is the reflection of $z$ with respect to the imaginary axis. This leads to the following.

\begin{lemma}\label{lem:symmetry}
 Let $w(z)$ be a weight function such that $w(z) = \overline{w(-\overline{z})}$ and $\Gamma$ be a contour that is invariant under reflection with respect to the imaginary axis, i.e.,
\[
 z \in \Gamma \Rightarrow -\overline{z} \in \Gamma.
\]
If a unique monic polynomial $p_n$ of degree $n$ exists that satisfies the orthogonality conditions
\[
 \int_\Gamma p_n(z) z^k w(z) dz = 0, \qquad k=0,\ldots,n-1,
\]
then
\begin{equation}\label{eq:symmetry}
 p_n(z) = (-1)^n \overline{p_n(-\overline{z})}.
\end{equation}
\end{lemma}
\begin{proof}
Using the symmetry of $\Gamma$ and of the weight function respectively, we find
\[
 \begin{aligned}
 \int_\Gamma p_n(z) z^k w(z) dz & = \int_\Gamma p_n(-\overline{z}) (-\overline{z})^k w(-\overline{z}) dz \\
 &= (-1)^k \int_\Gamma p_n(-\overline{z}) \overline{z^k} \, \overline{w(z)} dz \\
 &= (-1)^k \overline{\int_\Gamma \overline{p_n(-\overline{z})} z^k w(z) dz} = 0.
\end{aligned}
\]
Thus, $\overline{p_n(-\overline{z})}$ satisfies the same orthogonality conditions as $p_n(z)$. Since the latter is unique, we must have that
 \[
  \overline{p_n(-\overline{z})} = c p_n(z),
 \]
 for some constant $c$. Using that $p_n$ is monic yields $c=(-1)^n$.
\end{proof}
The weight function $e^{i \omega x}$ satisfies the required symmetry of Lemma \ref{lem:symmetry} and so does the interval $[-1,1]$. Therefore the polynomials $p_n^\omega$ satisfy the symmetry \eqref{eq:symmetry}. This implies among other things that they are either purely real or purely imaginary along the imaginary axis, depending on the parity of $n$.

\subsection{Derivatives with respect to $\omega$}

We gain useful knowledge about the polynomials by differentiating with respect to $\omega$. Surprisingly, this differentiation yields the orthogonal polynomial of smaller degree.

\begin{theorem}\label{th:derivatives}
 The derivative of $p_n^\omega$ with respect to $\omega$ satisfies
 \begin{equation}\label{eq:derivatives}
  \frac{\partial p_n^\omega}{\partial \omega}(x) = -i \beta_n \, p_{n-1}^\omega(x), \qquad n = 1, 2, \ldots,
 \end{equation}
 where the proportionality constant, which depends on $\omega$,
 \begin{equation}\label{eq:beta_n}
  \beta_n = \frac{ (p_n^\omega, p_n^\omega)}{(p_{n-1}^\omega, p_{n-1}^\omega)}
 \end{equation}
is one of the recurrence coefficients \eqref{eq:reccoeffs} of the orthogonal polynomials.
\end{theorem}
\begin{proof}
Define the quantities
\begin{equation}\label{eq:Fk}
F_k(\omega) = \int_{-1}^1 x^k p_n^\omega(x) e^{i \omega x} dx.
\end{equation}
Note that by orthogonality of the polynomials the first $n$ functions vanish identically,
\[
 F_k(\omega)=0, \qquad k=0,\ldots,n-1.
\]
Differentiating with respect to $\omega$ yields
\begin{equation}\label{eq:diff1}
 F_k'(\omega) = \int_{-1}^1 x^k \frac{\partial p_n^\omega}{\partial \omega}(x) e^{i \omega x} dx + i \int_{-1}^1x^{k+1} p_n^\omega(x) e^{i \omega x} dx =: G_k(\omega) + iF_{k+1}(\omega).
\end{equation}
Since $F_k'(\omega)=0$ for $k=0,\ldots,n-1$ and $F_{k+1}(\omega)=0$ for $k=0,\ldots,n-2$, we must have that
\[
G_k(\omega) =  \int_{-1}^1 x^k \frac{\partial p_n^\omega}{\partial \omega}(x) e^{i \omega x} dx = 0, \qquad k=0,\ldots,n-2.
\]
Because $p_n^\omega$ is monic, $\frac{\partial p_n^\omega}{\partial \omega}$ is a polynomial of degree at most $n-1$. By the relations above, it satisfies the exact same orthogonality conditions as $p_{n-1}^\omega$. Thus, there must be a constant depending on $\omega$, but not on $x$, such that
\[
 \frac{\partial p_n^\omega}{\partial \omega}(x) = c_n(\omega) p_{n-1}^\omega(x).
\]
It remains to determine this constant. Setting $k=n-1$ in \eqref{eq:diff1} we find
 \begin{equation}\label{eq:some_expression}
  F_{n-1}'(\omega) = G_{n-1}(\omega) + i F_n(\omega) = 0.
 \end{equation}
 Note from the orthogonality of $p_n^\omega$ that
 \[
  F_n(\omega) = \int_{-1}^1 x^n p_n^\omega(x) e^{i \omega x} dx = (p_n^\omega, p_n^\omega).
 \]
 Similarly, we find that
 \[
  G_{n-1}(\omega) = c_n(\omega) \int_{-1}^1 x^{n-1} p_{n-1}^\omega(x) e^{i \omega x} dx = c_n(\omega) (p_{n-1}^\omega, p_{n-1}^\omega).
 \]
 Combined with \eqref{eq:some_expression} this leads to the result.
\end{proof}

\subsection{Three-term recurrence relation}

As it turns out, the coefficients $\alpha_k$ and $\beta_k$ of the three-term recurrence relation \eqref{eq:recurrence} can be computed themselves by a recursion. The following result follows from taking the derivative with respect to $\omega$ of the recurrence relation. Note that $\alpha_k$ and $\beta_k$ are always functions of $\omega$. In the following we frequently omit this dependency and we denote the derivative of $\alpha_k$ with respect to $\omega$ simply by $\alpha_k'$ and we do so similarly for $\beta_k$.

\begin{theorem}\label{th:recurrence}
 Let $\alpha_k$ and $\beta_k$ be the $\omega$-dependent recurrence coefficients for $p_n^\omega$, defined by \eqref{eq:reccoeffs}. Then we have
 \[
  \beta_{k+1} = \beta_k - i \alpha_k'
 \]
 and
 \[
  \alpha_{k+1} = \alpha_k - i \frac{ \beta_{k+1}'}{\beta_{k+1}}.
 \]
\end{theorem}
\begin{proof}
Differentiating the recurrence \eqref{eq:recurrence} with respect to $\omega$ yields
\[
\frac{\partial p_{k+1}^\omega}{\partial \omega}(x)=(x-\alpha_k) \frac{\partial p_k^\omega}{\partial \omega}(x)-\alpha'_k p_k^\omega(x) -\beta'_k p_{k-1}^\omega(x)- \beta_k \frac{\partial p_{k-1}^\omega}{\partial \omega}(x).
\]
Using Thm. \ref{th:derivatives} and collecting terms leads to
\[
(-i\beta_{k+1}+\frac{\partial \alpha_k}{\partial \omega}) p_k^\omega(x)=((x-\alpha_k)(-i\beta_k)-\beta_k') p_{k-1}^\omega(x)+i\beta_{k-1} \beta_k p_{k-2}^\omega(x).
\]
Matching the leading order coefficients on both sides, using that both $p_k^\omega$ and $p_{k-1}^\omega$ are monic, we must have $-i\beta_{k+1}+\frac{\partial \alpha_k}{\partial \omega}=-i\beta_k$. The recurrence for $\beta_k$ follows from this. Dividing over yields
\[
p_k^\omega(x)=(x-\alpha_k-(-i\beta_k)^{-1} \beta_k' ) p_{k-1}^\omega(x)-\beta_{k-1}  p_{k-2}^\omega(x),
\]
and the recurrence for $\alpha_k$ follows by comparing to the regular recurrence for $p_k^\omega$,
\[
 p_k^\omega(x) = (x-\alpha_{k-1}) p_{k-1}^\omega(x)-\beta_{k-1}  p_{k-2}^\omega(x).
\]
\end{proof}

\begin{remark}
In the theory of orthogonal polynomials, these are sometimes known as the deformation equations for the recurrence coefficients. General expressions of this kind can be obtained whenever a weight function of exponential type is perturbed with a parameter, see for instance \cite[Prop. 2.1]{bleher2005asymptotics} and references therein.
\end{remark}

\subsection{Trajectories of the roots}\label{ss:trajectories}

Finally, we intend to show that the trajectories of the roots in the complex plane may have cusps. If so, the norm of the corresponding orthogonal polynomial vanishes, and as a result the orthogonal polynomial of one degree higher does not exist.

We start by showing the equivalence between the vanishing derivatives of the roots $x_j'$ and the vanishing derivatives of the polynomial $p_n^\omega$. As in the previous section we frequently omit the dependency of the roots on $\omega$ in our notation.

\begin{theorem}\label{thm:derivatives}
Assume that $p_n^\omega$ exists and let $x_j(\omega)$, $j=1,\hdots,n$, denote its zeros. If, for a given $\omega^*$, we have
\begin{equation}\label{eq:zerosder}
x_j'(\omega^*) = 0,\quad j=1,\hdots,n,
\end{equation}
then
\[
(p_n^{\omega^*},p_n^{\omega^*})=0.
\]
If $p_n^\omega$ is uniquely defined, then the converse is also true.
\end{theorem}
\begin{proof}
Writing $p_n^\omega$ in terms of its factors, we have
\[
p_n^\omega(x) = \prod_{j=1}^n (x-x_j(\omega) ).
\]
Differentiating with respect to $\omega$ yields
\begin{equation}\label{eq:diffxi}
\frac{\partial p_n^\omega(x)}{\partial \omega} = -\sum_{l=1}^n x_l'(\omega) \prod_{j=1,j\neq l}^n (x-x_j(\omega)).
\end{equation}
From this and condition \eqref{eq:zerosder} it follows that $\frac{\partial p_n^\omega}{\partial \omega}\Big{|}_{\omega=\omega^*}\equiv0$, i.e., the partial derivative of $p_n^\omega$ vanishes identically. By Theorem \ref{th:derivatives} we find that $\omega^*$ must be a root of $\beta_n(\omega)$, from which we conclude in turn that $(p_n^\omega,p_n^\omega)$ must vanish at $\omega=\omega^*$.

Recall the functions $F_k(\omega)$ from \eqref{eq:Fk} and their derivatives \eqref{eq:diff1}. Letting $k=n-1$, we find from the above and \eqref{eq:diff1} that
\[
 0=F_n(\omega^*)=\int_{-1}^1 p_n^{\omega^*}(x)x^n e^{i\omega^* x}{\rm d}x=(p_n^{\omega^*},p_n^{\omega^*}).
\]

To prove the converse, we assume $(p_n^{\omega^*},p_n^{\omega^*})=0$, which leads using \eqref{eq:diff1} again to
\[
\int_{-1}^1 \frac{\partial p_n^{\omega}(x)}{\partial \omega}\Big{|}_{\omega=\omega^*}x^ke^{i\omega^* x}{\rm d}x=0,\qquad k=0,\hdots,n-1.
\]
This means that $\frac{\partial p_n^{\omega}}{\partial \omega}\Big{|}_{\omega=\omega^*}$ satisfies the orthogonality conditions of $p_{n}^\omega$. However, it is a polynomial of degree $n-1$ and since $p_n^\omega$ exists uniquely, the lower-degree polynomial can only satisfy the above conditions if it vanishes identically.
\end{proof}

If $(p_n^{\omega^*},p_n^{\omega^*})=0$ for some value of $n$ and $\omega^*$, then the monic orthogonal polynomial $p_{n+1}^{\omega^*}$ of degree $n+1$ does not exist. In fact, the polynomial $p_n^{\omega^*}$ satisfies all of the orthogonality conditions that $p_{n+1}^{\omega^*}$ should satisfy, since
\[
 \int_{-1}^1 x^k p_n^{\omega^*}(x) e^{i \omega x} dx = 0, \qquad k=0,\ldots,n.
\]
This explains why the roots of $p_3^\omega$ agree with those of $p_2^\omega$ in Figure \ref{Fig:3Gausspoints} at specific values of $\omega$. The third root, as a function of $\omega$, has poles at these values of $\omega$.

\section{Properties of the quadrature rule}\label{S:4}

Since the weight function considered here is non-positive, only few of the classical results for orthogonal polynomials apply. We do not fully settle the questions of existence and uniqueness of the polynomials in this paper. Yet, several interesting properties of the quadrature rule can be established and we start with the asymptotic order.

\subsection{Asymptotic order}\label{ss:asymptotic}
The following theorem establishes a connection between polynomial accuracy and asymptotic accuracy, on the assumption that the quadrature points cluster near the endpoints. Note that a Gaussian quadrature rule with an odd number of points in our setting does not qualify here, since due to the symmetry at least one root is always on the imaginary axis. In the notation of the following theorem, a Gaussian rule with an even number of points corresponds to $N=4n$ and $M=2n$, and the theorem thus explains the error behaviour we observe in Fig. \ref{Fig:2Gausspoints_exp}.

\begin{theorem}\label{thm:asymptotic}
Let $\{x_j,w_j\}_{j=1}^{2n}$ be an $2n$-point quadrature rule with polynomial order $N$, i.e.,
\[
\sum_{j=1}^{2n} w_j x_j^k = \int_{-1}^1 x^k e^{i\omega x} {\rm d}x,\qquad k=0,1\hdots,N-1.
\]
and assume that $N \geq 2n$, i.e., the rule is at least interpolatory. 

If the nodes $x_j$ can be split in two groups $\{x_j^1\}_{j=1}^n$ and $\{x_j^{2}\}_{j=1}^n$, such that $x_j^1=-1+\mathcal{O}(\omega^{-1})$ and $x_j^2=1+\mathcal{O}(\omega^{-1})$, $j=1,\hdots,n$, then, for an analytic function $f$ the error has the asymptotic decay
\[
\sum_{j=1}^{2n} w_j f(x_j)-\int_{-1}^1 f(x)e^{i\omega x}{\rm d}x =\mathcal{O}(\omega^{-M-1}),
\]
where $M=\lfloor N/2 \rfloor$.
\end{theorem}

\begin{proof}
From \cite{2pointtaylor} we have that an analytic function $f$ can be written as
\[
f(x)=q_{2M-1}(x)+R(x)(x+1)^{M}(x-1)^{M},
\]
where $q_{2M-1}(x)$ is a polynomial of degree $2M-1\leq N-1$, and $R(x)$ is analytic. The error of the method is
\[
\sum_{j=1}^{2n} w_jf(x_j)-\int_{-1}^1 f(x)e^{i\omega x}{\rm d}x
\]
\[
=\sum_{j=1}^{2n} w_jR(x_j)(x_j+1)^{M}(x_j-1)^{M}-\int_{-1}^1 R(x)(x+1)^{M}(x-1)^{M}e^{i\omega x}{\rm d}x.
\]
By integration by parts the integral in this expression has asymptotic size $\mathcal{O}(\omega^{-M-1})$. The terms in the sum have a factor which is $\mathcal{O}(\omega^{-M})$, since $x_j=\pm 1+\mathcal{O}(\omega^{-1})$. It remains to show that $w_j=\mathcal{O}(\omega^{-1})$, $j=1,\hdots,2n$. Applying the method to the $j$-th Lagrange polynomial $l_j(x)$, we have, 
\[
w_j=\int_{-1}^1 l_j(x)e^{i\omega x}{\rm d}x = -\sum_{k=0}^{2n-1} \frac{l_j^{(k)}(x)}{(-i\omega)^{k+1}}e^{i\omega x}\Big{|}_{x=-1}^{x=1}.
\]
To conclude that $w_j=\mathcal{O}(\omega^{-1})$, we need that $l_j^{(k)}(\pm 1)=\mathcal{O}(\omega^k)$. Now assume the node $x_j$ is the first member of the group $1$: $x_j=x_{1}^1$. Writing $l_j(x)$ in terms of the two groups of nodes we have
\[
l_j(x) = \prod_{i=2}^n \frac{x-x_i^1}{x_{1}^1-x_i^1} \prod_{i=1}^n \frac{x-x_i^2}{x_{1}^1-x_i^2}.
\]
The second of these products is clearly bounded with its derivatives, and we can, by Leibniz's formula, concentrate on the first. Here, the denominator is of order $\mathcal{O}(\omega^{-n+1})$. Similarly, the numerator is of order $\mathcal{O}(\omega^{-n+1})$, when evaluated in $x=-1$. Differentiating will lower the degree of the numerator,  
\[
\frac{d}{dx}\prod_{i=2}^n (x-x_i^1)=\sum_{l=2}^n\prod_{i=2,i\neq l}^n(x-x_i^1).
\]
Evaluating in $x=-1$ gives that the numerator is of order $\mathcal{O}(\omega^{-n+2})$. In general the $k$-th derivative of the numerator is of order $\mathcal{O}(\omega^{-n+k+1})$, for $k\geq 1$, so
\[
\frac{l_j^{(k)}(x)}{(-i\omega)^{k+1}}e^{i\omega x}\Big{|}_{x=-1}^{x=1}=\mathcal{O}(\omega^{-1}).
\]
Similarly one shows that $l_j^{(k)}(- 1)=\mathcal{O}(\omega^k)$. The argument can be repeated for any $x_j$, regardless of which group it belongs to.
\end{proof}

\subsection{Large $\omega$ behaviour of $p_n^\omega$}

Next, a rather remarkable fact will be demonstrated, namely that pointwise we have that 
\begin{equation}\label{eq:limit}
p_{2n}^\omega(x)\rightarrow \left(\frac{i}\omega\right)^{2n}L_{n}(-i\omega (x+1))L_{n}(-i\omega (x-1)), \qquad \omega\to\infty,
\end{equation}
where $L_n(x)$ is the classical Laguerre polynomial of degree $n$. That is, the orthogonal polynomial becomes the product of two rescaled classical orthogonal polynomials in the limit $\omega\to\infty$. This is, in fact, the polynomial that vanishes at the superinterpolation points. To prove this we need the following intermediate result:

\begin{lemma}\label{thm:genV}
Consider a vector of $n>1$ points in $\mathbb{C}$, $x_1,\hdots,x_n$, and a sequence of integers $\alpha_1<\alpha_2<\hdots<\alpha_n$. We construct the generalised Vandermonde matrix $G=\{x_i^{\alpha_j}\}_{i,j=1}^n$,
\[
 G =
 \begin{pmatrix}
 x_1^{\alpha_1} & x_1^{\alpha_2} & \cdots & x_1^{\alpha_n}\\
 x_2^{\alpha_1} & x_2^{\alpha_2} & \cdots & x_2^{\alpha_n}\\
 \vdots &  \vdots & \ddots &   \vdots\\
  x_n^{\alpha_1} & x_n^{\alpha_2} & \cdots & x_n^{\alpha_n}\\
 \end{pmatrix}
\]
Then 
\[
{\rm det}(G)=L(x_1,x_2,\hdots,x_n)\prod_{1\leq i<j\leq n} (x_i-x_j),
\]
where $L(x_1,x_2,\hdots,x_n)$ is a polynomial in $x_1,\hdots,x_n$.
\end{lemma}
\begin{proof}
Consider the determinant as a function of $x_1,\hdots,x_n$,
\[
{\rm det}(G)=H(x_1,\hdots,x_n).
\]
By expanding the determinant along any row, it is apparent that $H(x_1,\hdots,x_n)$ is actually a polynomial.
Now
\[
H(x_j,x_2,\hdots,x_n)=0,\qquad j=2,3,\hdots,n,
\]
since this corresponds to a matrix with duplicate rows. This implies that $x_1-x_j$, $j=2,\hdots,n$ is a factor of $H$. Similarly, if one considers $H$ in terms of the remaining arguments, one sees that $x_i-x_j$, $i\neq j$, are all factors of $H$. The result follows from this.
\end{proof}

The key to showing something like \eqref{eq:limit}, is to look at $p_n^\omega$ evaluated in the superinterpolation points.

\begin{lemma}\label{thm:zeros}
Assume $p_{2n}^\omega$ is bounded in $\omega$ together with all its derivatives. Let $x_j$, $j=1,\hdots,2n$, denote the superinterpolation points. Then 
\[
p_{2n}(x_j)=\mathcal{O}(\omega^{-n-1}), \qquad j=1,\hdots,2n.
\]
\end{lemma}
\begin{proof}
The orthogonality condition for $p_{2n}^\omega$ is
\[
\int_{-1}^1p_{2n}(x)x^je^{i\omega x}{\rm d}x = 0, \qquad j=0,\hdots,2n-1.
\]
Applying the interpolatory quadrature rule based on interpolation at the superinterpolation points, along with the assumption on the boundedness of $p_{2n}^\omega$, gives \cite[Th.3.2]{Huybrechs:2012gm}
\[
\sum_{i=1}^{2n}w_i p_{2n}(x_i)x_i^j=\mathcal{O}(\omega^{-2n-1}), \qquad j=0,\hdots,2n-1.
\]
Here the weights are given in terms of the Gauss-Laguerre weights $\eta_j$,
\begin{equation}\label{eq:w}
w_j=w_{j+n}=\frac{\eta_j}{\omega}.
\end{equation}
Denoting $y=[w_1p_{2n}^\omega(x_1),w_2p_{2n}^\omega(x_2),\hdots,w_{2n}p_{2n}^\omega(x_{2n})]^T$, this is the Vandermonde system
\begin{equation}\label{eq:V}
Vy=\mathcal{O}(\omega^{-2n-1}),
\end{equation}
where $V=\{x_i^{j-1}\}_{i,j=1}^{2n}$. By Cramer's rule we have
\[
V^{-1}=\frac{{\rm Adj}(V)}{{\rm det}(V)}.
\]
Entries of ${{\rm Adj}(V)}$ are computed in terms of determinants: 
\[
{{\rm Adj}(V)}_{i,j}=(-1)^{i+j}{\rm det}(V^{j,i}),
\]
where $V^{j,i}$ is $V$ with row $j$ and column $i$ deleted. Using this we can find the asymptotics of $V^{-1}$ elementwise. First, using the well known formula for the Vandermonde determinant,
\[
{\rm det}(V)=\prod_{1\leq i,j \leq 2n}(x_i-x_j) = C\omega^{-n(n-1)}+\mathcal{O}(\omega^{-n(n-1)-1}),
\]
where $C\neq 0$, which follows from the fact that the points belong two groups of $n$ points which are $\omega^{-1}$ close to either $-1$ or $1$. Similarly, if we delete one row and one column of $V$, one of the groups will be missing a point, and $V^{j,i}$ will be a generalised Vandermonde matrix. The form of the determinant is given by Lemma \ref{thm:genV}, and from this it follows that 
\[
{\rm det}(V^{j,i}) = \mathcal{O}(\omega^{-(n-1)^2}), \qquad 0<i,j\leq 2n.
\]
From this it follows that $V^{-1}= \mathcal{O}(\omega^{n-1})$ element wise. This further gives, from \eqref{eq:V}, that $y=\mathcal{O}(\omega^{-n-2})$, elementwise, and thus, by \eqref{eq:w}, the desired result.
\end{proof}

Finally we can state the correspondence between roots of $p_{2n}^\omega$ and superinterpolation points, $\omega\to\infty$.

\begin{theorem}\label{thm:roots}
Let $x_j$, $j=1,\hdots,2n$ denote the $2n$ superinterpolation points. Assume $p_{2n}^\omega$ exists, and that it is bounded with all its derivatives in $\omega$.
For each $y_j$ such that $p_{2n}^\omega(y_j)=0$, $j=1,\hdots,2n$, there is a corresponding index $l$ such that 
\[
y_j = x_l + {\mathcal O}(\omega^{-2}) ,\qquad \omega \to\infty.
\]
\end{theorem}
\begin{proof}
Writing $p_{2n}^\omega$ in terms of its factors, we have
\[
p_{2m}^\omega(x)=\prod_{j=1}^{2n}(x-y_j).
\]
From Lemma \ref{thm:zeros} we have that 
\[
\prod_{j=1}^{2n}(x_k-y_j)=\mathcal{O}(\omega^{-n-1}), \qquad k=1,\hdots,2n.
\]
Or, in terms of the Gauss-Laguerre points,
\[
\prod_{j=1}^{n}(\pm 1+\frac{i\xi_k}{\omega}-y_j)(\mp 1+\frac{i\xi_k}{\omega}-y_j)=\mathcal{O}(\omega^{-n-1}), \qquad k=1,\hdots,n.
\]
From this expression, it is clear that for each $j=1,\hdots,2n$ we must have $y_j\sim \pm1$, which gives $\mathcal{O}(\omega^{-n})$. The only way $\mathcal{O}(\omega^{-n-1})$ can be attained is if 
\[
y_j = \frac{i \xi_l}{\omega} + {\mathcal O}(\omega^{-2}) = x_l + {\mathcal O}(\omega^{-2}),
\]
for some $l$.
\end{proof}

\begin{corollary}
Assume $p_{2n}^\omega$ exists, and that it is bounded with all its derivatives in $\omega$. The $2n$ point Gaussian rule, with nodes being the zeros of $p_{2n}^\omega$, is of asymptotic order $n+1$.
\end{corollary}
\begin{proof}
From Thm. \ref{thm:roots} we have that all roots behave like $\pm1 +\mathcal{O}(\omega^{-1})$, divided into two equally sized groups. The result thus follows from Thm. \ref{thm:asymptotic}.
\end{proof}

Note that in the last two results we assumed that the polynomial $p_{2n}^\omega$ is bounded in $\omega$, along with all its derivatives with respect to $x$. This, along with the assumption that the polynomials exist for all $\omega$, is not proved in this paper.

\section{Conclusions and outlook}

We have presented a Gaussian quadrature rule for integrals on $[-1,1]$ with weight function $e^{i\omega x}$, $\omega\geq 0$. The associated family of orthogonal polynomials $p_n^\omega$ is non standard
because of the oscillatory nature of the weight function. It is shown that for some critical values of $\omega$, the orthogonal polynomials of odd degree fail to exist, and that phenomenon is related to the fact 
that the bilinear form defined with respect to $e^{i\omega x}$ is not positive definite. However, based on numerical experiments we conjecture that all polynomials of even degree exist 
for any value of $\omega$. We give some results on this family of orthogonal polynomials, but their global behaviour in terms of $\omega$ is bound to be very complicated. We note that they should bridge 
between the Legendre case (when $\omega=0$) and a product of two rotated and scaled Laguerre polynomials in the complex plane, as $\omega\to\infty$. 

We also present results connecting the concept of standard polynomial accuracy of Gaussian quadrature rules and that of asymptotic order in terms of $\omega$, which is of great interest 
in the context of highly oscillatory quadrature. These two ideas are related under the assumption that the zeros of $p_n^\omega$ cluster near the endpoints, something that is observed numerically in the paper.

\section*{Acknowledgements}
Arieh Iserles has on several occasions suggested the authors having a look at the problem of Gaussian quadrature rules for oscillatory integrals over bounded domains. It is an honour to dedicate this work to Arieh on the occasion of his 65th birthday.

\bibliographystyle{plain}
\bibliography{bibliography}

\end{document}